\documentclass[11pt,leqno]{article}
\usepackage[utf8]{inputenc}
\usepackage{amsmath,amssymb,amsthm,latexsym}
\usepackage{indentfirst}
\usepackage{amsmath}
\newtheorem{definition}{\bf \large Definition}[section]
\newtheorem{theorem}{\bf \large Theorem}[section]
\newtheorem{PROPOSITION}{\bf \large Proposition}[section]
\newtheorem{corollary}{\bf \large Corollary}[section]
\newtheorem{lemma}{\bf \large Lemma}[section]

\newtheorem{remark}{\bf \large Remark}[section]

\newtheorem{conjecture}{\bf \large CPE Conjecture}

\title{Rigidity of Critical Point Metrics under some Ricci curvature constraints}
\author {\normalsize {Tongzhu Li, ~~Junlong Yu}\\
\small{Department of Mathematics, Beijing institute of
technology,}\\
\small{Beijing, 100081, China.}\\
\small{E-mail:~litz@bit.edu.cn,~~~3120245998@bit.edu.cn.}}

\date{}

\begin{document}
\maketitle
\begin{abstract}
A critical point metric is a critical point of the total scalar curvature functional
restricted to the space of constant scalar curvature metrics  on a closed
manifold with unit volume. It was conjectured in 1980's that every critical point metric must be Einstein.
In this paper, we prove that this conjecture is true if   the norm of the traceless   Ricci operator $|\widetilde{Ric}|$ is constant.
For $3$-dimensional case, we prove that the conjecture is  true,
 if  the traceless   Ricci operator  satisfies $tr((\widetilde{Ric})^3)\geq -\frac{R}{12}|\widetilde{Ric}|^2$, where $R$ denotes the scalar curvature.
\end{abstract}


\section{Introduction}
The classic method for understanding constant curvature metric is to  investigate the properties of the critical point of  the total scalar curvature functional under appropriate  constraints.
Let $(M^n,g)$ be a connected closed(compact without boundary) orientable  smooth $n$-dimensional Riemannian manifold, and $\mathfrak{M_1}$ be the space of all smooth metrics with unit volume on $M^n$. The scalar curvature $R$ can be considered as a nonlinear map $R: \mathfrak{M_1} \rightarrow C^\infty (M^n).$
The well-known Hilbert-Einstein functional is defined as
\begin{equation*}
F(g) = \int_{M^n} R_g dv_g.
\end{equation*}
where $R_g$ is the scalar curvature of the metric $g\in \mathfrak{M_1}$ and $dv_g$ denotes the volume form of
$g$. It is well known that the formal $L^2$-adjoint of the linearization of the scalar
curvature operator $R_g$ at $g$ is defined as
$$\mathfrak{R}^*_g(f)=-(\triangle f)g+Hess_g(f)-fRic_g,$$
where $f$ is a smooth function on $M^n$, and $\triangle, Hess_g$, and $Ric_g$ denote  the Laplacian, the Hessian operator, and the Ricci
curvature tensor on $M^n$, respectively.
Let $$\widetilde{Ric}_g=Ric_g-\frac{R_g}{n}g$$ the traceless part of Ricci tensor,
then the Euler-Lagrange equation  of $F(g)$ ( see \cite{Besse})  is given by
\begin{equation}\label{euler}
\mathfrak{R}^*_g(f)=\widetilde{Ric}_g
\end{equation}

Taking the trace in (\ref{euler}), we obtain
$$\triangle f+\frac{R}{n-1}f=0.$$
Thus the nonconstant function $f$ is an eigenfunction of the Laplacian and then $R$ is positive constant.

Clearly the equation (\ref{euler}) can be rewritten as
\begin{equation}\label{euler1}
\nabla^2 f =(1+f)\widetilde{Ric}_g  - \frac{R}{n(n-1)}fg.
\end{equation}

\begin{definition}
A critical point  metric (or \emph{CPE metric}) is a three-tuple $(M^n, g,f), (n\geq 3)$, where $(M^n, g)$ is a closed (compact without boundary), oriented $n$-dimensional Riemannian manifold with constant scalar curvature, and $f: M^n\to \mathbb{R}$ is a nonconstant smooth function satisfying
equation (\ref{euler}) (or \ref{euler1}). Such a function $f$ is called a potential function.
\end{definition}

Obviously CPE metric is Einstein  when $f=0$.
Moreover, the existence of a non constant solution is only known in the round
sphere for some height function. In 1987, A.Besse proposed a conjecture in the book (p.128 in \cite{Besse}):
\begin{conjecture}
 CPE metrics are Einstein.
\end{conjecture}
Moreover, Besse showed the conjecture is true, if $$\frac{R_g}{n-1} \not\in Spec(- \Delta_g).$$ Otherwise, he conjectured the CPE metric should be a round sphere, since it was the only case known back then.
Equation (\ref{euler}) or (\ref{euler1}) is usually referred as \emph{Critical Point Equation} (\emph{CPE}) in the study of \emph{CPE Conjecture}.

As an corollary of \emph{CPE Conjecture}, $(M^n,g)$ has to be a round sphere in this case by the well-known \emph{Obata's Theorem} (\cite{Obata}):

\begin{theorem}[M.Obata \cite{Obata}]
For $n \geq 2$, let $(M^n,g)$ be a connected complete Riemannian manifold. Suppose there is a smooth function $f \not\equiv 0$ solves the following equation
\begin{equation}
\nabla^2 f + f g = 0.
\end{equation}
Then $(M,g)$ is isometric to the unit round sphere $\mathbb{S}^n$.
\end{theorem}

Now we can give a refined version of \emph{CPE Conjecture}:
\begin{conjecture}
A \emph{CPE metric} is spherical.
\end{conjecture}

It is clear that $(M^n, g)$ is Einstein if it admits a trivial solution $f\equiv -1$. Other CPE metrics with constant function $f$ are
Ricci flat metrics. $g$ is isometric to a round sphere metric if it is an Einstein CPE metric with a non-constant function $f$. Hence
Conjecture 1.2 really says that a CPE metric with a non-constant solution $f$ to (\ref{euler1}) is isometric to a round sphere metric.

After being proposed, many attempts had been tried to solve this conjecture. The conjecture was verified with various
additional presumptions. The first result known now is due to J. Lafontaine. He proved \emph{CPE Conjecture} is true for a  locally conformally flat \emph{CPE metric} (cf. \cite{Lafontaine}, \cite{Q-Y}). In particular, J. Chang, S. Hwang and G. Yun did a series study on \emph{CPE metrics} (\cite{C-H}-\cite{Hwang_3}), which gave us fundamental interpretations on the geometric structures of \emph{CPE metrics}. For example, S. Hwang first discovered that whether $\min_{x \in M^n} f(x) = -1 $  plays a very essential role in telling a \emph{CPE metric} is spherical (\cite{Hwang_1}). Besides posing presumptions on $f$, they also achieved many partial results when introducing additional assumptions on curvatures. In \cite{C-H-Y_1}, they showed that a \emph{CPE metric} has to be spherical if the Ricci tensor is parallel. In 2011, they proved a \emph{CPE metric} of harmonic curvature must be spherical by a very careful examination on the geometric structure of the level set $\{x\in M : f(x) = -1\}$ with the assistance of some geometric equalities. As a corollary, the conjecture is verified for locally conformally flat \emph{CPE metric} without assuming it is vacuum static (\cite{C-H-Y_3}).

Recently, A.Barros and E.Ribeiro Jr verified \emph{CPE Conjecture} for a four dimensional half locally conformally flat \emph{CPE metric} (\cite{B-R}).  J. Qing and W. Yuan showed a CPE metric has to be spherical, if it is Bach flat or even weaker presumptions if dimension is three (\cite{Q-Y}). In \cite{bal-0}, H. Baltazar, A. Barros, R. Batista and E. Viana showed a CPE metric has to be spherical, if it has zero radial Weyl curvature (i.e., $i_{\nabla f}W=0$). When the dimension $n=3$, Baltazar \cite{Bal} proved that the CPE conjecture is true for three-dimensional manifolds with
nonnegative sectional curvature. In \cite{He}, H. Y. He proved that the CPE conjecture is true for three-dimensional manifolds with
nonnegative Ricci curvature. For more references on CPE
metrics, see \cite{bal-1}, \cite{F_Y}, \cite{C-H}, \cite{C-H-Y_1}, \cite{C-H-Y_2},  \cite{C-H-Y_3-1}, \cite{C-H-Y_3}, \cite{bl} and references therein.

We note that \emph{CPE Conjecture} is true if and only if the traceless   Ricci operator $\widetilde{Ric}=0$. In this paper,
we consider CPE Conjecture under  some conditions related to the traceless   Ricci operators $\widetilde{Ric}$. Our results are as follows,

\begin{theorem}\label{th-1}
Let $(M^n, g, f) (n\geq3)$ be a closed, oriented, connected CPE metric. If there exists an integer $k (k\geq 0)$ such that
$$\int_{M^n}(1+f)^{2k}\widetilde{Ric}(\nabla f,\nabla f)dv_g\geq 0,$$
then \emph{CPE Conjecture} is true, \emph{i.e.} a \emph{CPE} metric has to be spherical.
\end{theorem}

\begin{corollary}\label{co-1}
Let $(M^n, g, f) (n\geq 3)$ be a closed, oriented, connected CPE metric. If
$$\int_{M^n}\widetilde{Ric}(\nabla f,\nabla f)dv_g\geq 0,$$
then \emph{CPE Conjecture} is true, \emph{i.e.} a \emph{CPE} metric has to be spherical.
\end{corollary}

\begin{theorem}\label{th-2}
Let $(M^n, g, f) (n\geq 3)$ be a closed, oriented, connected CPE metric. If the norm of the traceless Ricci curvature is constant, that is
$$|\widetilde{Ric}|^2=constant,$$
then $\widetilde{Ric}=0$ and \emph{CPE Conjecture} is true, \emph{i.e.} a \emph{CPE} metric has to be spherical.
\end{theorem}

When the dimension $n=3$,  we obtain the following theorems,
\begin{theorem}\label{th-3-1}
Let $(M^3, g, f)$ be a $3$-dimensional closed, oriented, connected CPE  metric. If
$$tr((\widetilde{Ric})^3)\geq \frac{-R}{12}|\widetilde{Ric}|^2,$$
then, $\widetilde{Ric}=0$  and $(M^3, g, f)$is isometric to a round sphere $\mathbb{S}^3$.
\end{theorem}
\begin{theorem}\label{th-4-1}
Let $(M^3, g, f)$ be a $3$-dimensional closed, oriented, connected CPE  metric. If
$$|\widetilde{Ric}|^2\leq \frac{R^2}{24},$$
then, $\widetilde{Ric}=0$  and $(M^3, g, f)$is isometric to a round sphere $\mathbb{S}^3$.
\end{theorem}
\begin{theorem}\label{th-3-2}
Let $(M^3, g, f)$ be a $3$-dimensional closed, oriented, connected CPE  metric. If
$$\frac{-5R}{24}|\widetilde{Ric}|^2\leq tr((\widetilde{Ric})^3)\leq 0,$$
then, $\widetilde{Ric}=0$  and $(M^3, g, f)$is isometric to a round sphere $\mathbb{S}^3$.
\end{theorem}

We organize the paper as follows. In section 2, we established integral identities about the traceless   Ricci operator $\widetilde{Ric}$ on $n$-dimensional CPE manifolds, and then used them to prove our theorems \ref{th-1} and \ref{th-2}.  In section 3, by the particularity of three-dimensional Riemannian manifold, we established  integral identities about the traceless   Ricci operator  on $3$-dimensional CPE manifolds, and then used them to prove our theorem \ref{th-3-1}, theorem \ref{th-4-1} and theorem \ref{th-3-2}.

\par\noindent
\section{Integral identities on $n$-dimensional CPE manifold and proof of Theorems}
Let $(M^n,g) (n\geq 3)$ be an $n$-dimensional  orientable Riemannian manifold with the Levi-Civita connection $\nabla$. For any point $p\in M^n$, we  can choose $\{e_1,\cdots,e_n\}$ as a local orthonormal frame field at $p$.
 Let $R_{ijkl}, R_{ij}$ denote the coefficients
of the Riemannian curvature tensor
and the Ricci curvature tensor, respectively.  The decomposition of the Riemannian curvature tensor into irreducible components yields
\begin{equation}\label{cur1}
\begin{split}
R_{ijkl}&=W_{ijkl}+\frac{1}{n-2}\Big(R_{ik}\delta_{jl}+R_{jl}\delta_{ik}-R_{jk}\delta_{il}-R_{il}\delta_{jk}\Big)\\
&-\frac{R}{(n-1)(n-2)}(\delta_{ik}\delta_{jl}-\delta_{jk}\delta_{il}),
\end{split}
\end{equation}
where $W_{ijkl}$ are the components of the Weyl tensor and
$R:=tr(Ric)$ is the scalar curvature of $M^n$ (cf \cite{Koiso}).

The  traceless   Ricci operator $\widetilde{Ric}$ defined as
$\widetilde{Ric}=Ric-\frac{R}{n}g.$
Let  $\overset{\circ}{R}_{ij}=R_{ij}-\frac{R}{n}\delta_{ij}$ be the coefficients of  the traceless   Ricci operator $\widetilde{Ric}$,
and the equation (\ref{cur1})  can be rewritten as the following equation
\begin{equation}\label{cur1-0}
\begin{split}
R_{ijkl}&=W_{ijkl}+\frac{1}{n-2}\Big(\overset{\circ}{R}_{ik}\delta_{jl}
+\overset{\circ}{R}_{jl}\delta_{ik}-\overset{\circ}{R}_{jk}\delta_{il}-\overset{\circ}{R}_{il}\delta_{jk}\Big)\\
&+\frac{R}{n(n-1)}(\delta_{ik}\delta_{jl}-\delta_{jk}\delta_{il}),
\end{split}
\end{equation}

For a smooth function $f\in C^{\infty}(M^n)$,
Let
$$\nabla_{e_i}f=f_i, ~\nabla^2_{e_j,e_i}f=f_{ij}, \nabla^3_{e_i,e_j,e_k}f=f_{ijk}, \cdots,$$
then we have the following Ricci identities,
\begin{equation}\label{ricci}
f_{ijk}-f_{ikj}=\sum_mf_{m}R_{mijk},~~f_{ijkl}-f_{ijlk}=\sum_mf_{mj}R_{mikl}+\sum_mf_{mi}R_{mjkl}.
\end{equation}
In this section we  use
the following range of indices: $1\leq
i,j\leq n$.

Using the contracted second Bianchi identy, we have
\begin{equation}\label{sca}
\sum_jR_{ij,j}=\sum_j\nabla_jR_{ij}=\frac{1}{2}e_i(R)=\frac{1}{2}R_i,~~\sum_j\overset{\circ}{R}_{ij,j}=\frac{n-2}{2n}R_i.
\end{equation}

Now we assume that $(M^n, g)$ is a CPE metric with the potential function $f$, then
\begin{equation}\label{cpe-1}
f_{ij}=(1+f)\overset{\circ}{R}_{ij}-\frac{R}{n(n-1)}f\delta_{ij}.
\end{equation}

By taking the trace of (\ref{cpe-1}), we get
\begin{equation}\label{eig}
\Delta f=\frac{-R}{n-1}f.
\end{equation}
We assume that $f$ is not constant, thus $f$ is an eigenfunction of the Laplacian operator and the scalar curvature $R$  must be a positive constant.
The equation (\ref{sca}) reduces to
\begin{equation}\label{sca-1}
R_{ij,k}=\overset{\circ}{R}_{ij,k},~~\sum_jR_{ij,j}=0,~~\sum_j\overset{\circ}{R}_{ij,j}=0.
\end{equation}
From (\ref{cpe-1}),
\begin{equation}\label{cpe-2}
f_{ij,k}=(1+f)\overset{\circ}{R}_{ij,k}+(\overset{\circ}{R}_{ij}-\frac{R}{n(n-1)}\delta_{ij})f_k.
\end{equation}
Using the Ricci identity and (\ref{cpe-2}), we can get
\begin{equation}\label{cpe-4}
\begin{split}
&(1+f)\overset{\circ}{R}_{ij,k}-(1+f)\overset{\circ}{R}_{ik,j}\\
&=\sum_mf_mR_{mijk}+f_j\overset{\circ}{R}_{ik}-f_k\overset{\circ}{R}_{ij}+\frac{R}{n(n-1)}(f_k\delta_{ij}-f_j\delta_{ik}).
\end{split}
\end{equation}

Next we use the divergence theorem to establish some geometric integral identities. Then we use these integral identities to prove our Theorem \ref{th-1} and Theorem \ref{th-2}
\begin{PROPOSITION}\label{pr-1}
Let $(M^n, g, f) (n\geq 3)$ be a closed CPE metric, then
\begin{equation}\label{int-1}
\int_{M^n}\Big[k(1+f)^{k-1}\widetilde{Ric}(\nabla f,\nabla f)+(1+f)^{k+1}|\widetilde{Ric}|^2\Big]dv_g=0,
\end{equation}
where $k$ is a nonnegative integer $(k\geq 0)$.
\end{PROPOSITION}
\begin{proof}
For a nonnegative integer $k\geq 0$, define the vector field:
$$Z=\sum_{ij}(1+f)^k\overset{\circ}{R}_{ij}f_ie_j.$$
Using (\ref{sca-1}), and taking the divergence of $Z$,
\begin{equation*}
\begin{split}
div(Z)&=\sum_{ij}(1+f)^k\Big[\overset{\circ}{R}_{ij,i}f_i+\overset{\circ}{R}_{ij}f_{ij}\Big]+\sum_{ij}k(1+f)^{k-1}\overset{\circ}{R}_{ij}f_if_j\\
&=\sum_{ij}(1+f)^k\overset{\circ}{R}_{ij}f_{ij}+\sum_{ij}k(1+f)^{k-1}\overset{\circ}{R}_{ij}f_if_j\\
&=k(1+f)^{k-1}\widetilde{Ric}(\nabla f,\nabla f)+(1+f)^{k+1}|\widetilde{Ric}|^2.
\end{split}
\end{equation*}
Integrating the equation on both sides,  we complete the proof of Proposition \ref{pr-1}.
\end{proof}

\begin{PROPOSITION}\label{pr-2}
Let $(M^n, g, f)$ be a closed CPE metric, then
\begin{equation}\label{int-2}
\int_{M^n}\Big[\widetilde{Ric}(\nabla f,\nabla f)+(1+f)f|\widetilde{Ric}|^2\Big]dv_g=0.
\end{equation}
\end{PROPOSITION}
\begin{proof}
Define the vector field:
$$Z=\sum_{ij}f\overset{\circ}{R}_{ij}f_ie_j.$$
Using (\ref{sca-1}), and taking the divergence of $Z$,
\begin{equation*}
\begin{split}
div(Z)&=\sum_{ij}f\Big[\overset{\circ}{R}_{ij,i}f_i+\overset{\circ}{R}_{ij}f_{ij}\Big]+\sum_{ij}\overset{\circ}{R}_{ij}f_if_j\\
&=\sum_{ij}f\overset{\circ}{R}_{ij}f_{ij}+\sum_{ij}\overset{\circ}{R}_{ij}f_if_j\\
&=\widetilde{Ric}(\nabla f,\nabla f)+(1+f)f|\widetilde{Ric}|^2.
\end{split}
\end{equation*}
Integrating the equation on both sides,  we complete the proof of Proposition \ref{pr-2}.
\end{proof}

\begin{PROPOSITION}\label{pr-3}
Let $(M^n, g, f)$ be a closed CPE metric, then
\begin{equation}\label{int-3}
\int_{M^n}(1+f)|\widetilde{Ric}|^2dv_g=0.
\end{equation}
\end{PROPOSITION}
\begin{proof}
In equation (\ref{int-1}), taking $k=1$, we have
\begin{equation}\label{int-4}
\int_{M^n}\Big[\widetilde{Ric}(\nabla f,\nabla f)+(1+f)^2|\widetilde{Ric}|^2\Big]dv_g=0
\end{equation}
The equation (\ref{int-2}) minus the equation (\ref{int-4}), we obtain the equation (\ref{int-3}).
\end{proof}
\begin{remark}
The equation (\ref{int-3})  is also obtained in paper \cite{Hwang_1}.
\end{remark}

Now we prove Theorem \ref{th-1}. In (\ref{int-1}), taking integer $2k+1$, we have
\begin{equation}\label{int-1-0}
\int_{M^n}\Big[(2k+1)(1+f)^{2k}\widetilde{Ric}(\nabla f,\nabla f)+(1+f)^{2(k+1)}|\widetilde{Ric}|^2\Big]dv_g=0,
\end{equation}

If $\int_{M^n}(1+f)^{2k}\widetilde{Ric}(\nabla f,\nabla f)dv_g\geq0$, then from (\ref{int-1-0})
$$\int_{M^n}(1+f)^{2(k+1)}|\widetilde{Ric}|^2dv_g\leq 0,$$
which implies that $|\widetilde{Ric}|^2=0$. Thus the \emph{CPE} metric is Einstein and Theorem \ref{th-1} is proved.

In theorem \ref{th-1}, let $k=0$,  then we prove  Corollary \ref{co-1}.

Now we prove Theorem \ref{th-2}.
From (\ref{int-3}), we get
\begin{equation*}
0=\int_{M^n}(1+f)|\widetilde{Ric}|^2dv_g=|\widetilde{Ric}|^2\int_{M^n}(1+f)dv_g=|\widetilde{Ric}|^2\int_{M^n}dv_g,
\end{equation*}
Which implies that$|\widetilde{Ric}|=0$, Thus the \emph{CPE} metric is Einstein and Theorem \ref{th-2} is proved.

\par\noindent
\section{Integral identities on $3$-dimensional CPE manifold and proof of Theorems}

For $3$-dimensional Riemannian  manifold $(M^3, g)$, the Weyl curvature tensor  is always zero, thus the Ricci curvature determines the
full curvature tensor,
\begin{equation}\label{cur1-1}
R_{ijkl}=R_{ik}\delta_{jl}+R_{jl}\delta_{ik}-R_{jk}\delta_{il}-R_{il}\delta_{jk}
-\frac{R}{2}(\delta_{ik}\delta_{jl}-\delta_{jk}\delta_{il}).
\end{equation}
On $3$-dimensional CPE manifold, the CPE equation is given by
\begin{equation}\label{cpe-1-1}
f_{ij}=(1+f)\overset{\circ}{R}_{ij}-\frac{R}{6}f\delta_{ij}.
\end{equation}
Furthermore, we have the following equation,
\begin{equation}\label{cpe-2}
f_{ij,k}=(1+f)\overset{\circ}{R}_{ij,k}+(\overset{\circ}{R}_{ij}-\frac{R}{6}\delta_{ij})f_k.
\end{equation}

Form (\ref{cpe-4}),
\begin{equation}\label{cpe-4-1}
\begin{split}
(1+f)\overset{\circ}{R}_{ij,k}-(1+f)\overset{\circ}{R}_{ik,j}
=f_m\overset{\circ}{R}_{mj}\delta_{ik}-f_m\overset{\circ}{R}_{mk}\delta_{ij}
+2(f_j\overset{\circ}{R}_{ik}-f_k\overset{\circ}{R}_{ij}).
\end{split}
\end{equation}

Let $\overline{R}_{ij}=(1+f)\overset{\circ}{R}_{ij}$, then
$$\overline{R}_{ij,k}=(1+f)\overset{\circ}{R}_{ij,k}+f_k\overset{\circ}{R}_{ij}.$$
From (\ref{cpe-4-1}),  we obtain   the following equation,
\begin{equation}\label{ricci-2}
\overline{R}_{ij,k}-\overline{R}_{ik,j}
=\sum_mf_m\overset{\circ}{R}_{mj}\delta_{ik}-\sum_mf_m\overset{\circ}{R}_{mk}\delta_{ij}
+f_j\overset{\circ}{R}_{ik}-f_k\overset{\circ}{R}_{ij}.
\end{equation}
In this section we  use
the following range of indices: $1\leq
i,j\leq 3$.

Next, we  use the particularity of three-dimensional manifolds to establish more geometric integral identities.

\begin{PROPOSITION}\label{pr-4}
Let $(M^n, g, f)$ be a $3$-dimensional closed, oriented, connected CPE  metric, then
\begin{equation*}
\int_{M^n}\Big[\frac{3}{2}|\widetilde{Ric}|^2|\nabla f|^2
+\frac{R}{12}(1+f)^2|\widetilde{Ric}|^2
-2\sum_{ijk}\overset{\circ}{R}_{ij}\overset{\circ}{R}_{jk}f_if_k+(1+f)^2tr((\widetilde{Ric})^3)\Big]dv_g=0.
\end{equation*}
\end{PROPOSITION}
\begin{proof}
Define the vector field:
$$Z=\sum_{ij}\overset{\circ}{R}_{ij}\bar{R}_{jk}f_ke_i.$$
Using (\ref{sca-1}), and taking the divergence of $Z$,
\begin{equation*}
\begin{split}
div(Z)&=\sum_{ij}\Big[\overset{\circ}{R}_{ij,i}\bar{R}_{jk}f_k+\overset{\circ}{R}_{ij}\bar{R}_{jk,i}f_k+\overset{\circ}{R}_{ij}\bar{R}_{jk}f_{ki}\Big]\\
&=\sum_{ij}\Big[\overset{\circ}{R}_{ij}(\bar{R}_{ij,k}+\bar{R}_{jk,i}-\bar{R}_{ij,k})f_k+\overset{\circ}{R}_{ij}\bar{R}_{jk}f_{ki}\Big].\\
\end{split}
\end{equation*}
Using (\ref{ricci-2}), the first term
\begin{equation*}
\begin{split}
&\sum_{ij}\Big[\overset{\circ}{R}_{ij}(\bar{R}_{ij,k}+\bar{R}_{jk,i}-\bar{R}_{ij,k})f_k\Big]\\
&=\sum_{ij}\Big[\overset{\circ}{R}_{ij}(f_k\overset{\circ}{R}_{ij}+(1+f)\overset{\circ}{R}_{ij,k})f_k
+\overset{\circ}{R}_{ij}(\bar{R}_{jk,i}-\bar{R}_{ij,k})f_k\Big]\\
&=|\widetilde{Ric}|^2|\nabla f|^2+\sum_k\frac{1}{2}\nabla_k|\widetilde{Ric}|^2(1+f)f_k
+|\widetilde{Ric}|^2|\nabla f|^2-2\sum_{ijk}\overset{\circ}{R}_{ij}\overset{\circ}{R}_{jk}f_if_k.
\end{split}
\end{equation*}
Using the CPE equation (\ref{cpe-1-1}), the second term
\begin{equation*}
\sum_{ij}\Big[\overset{\circ}{R}_{ij}\bar{R}_{jk}f_{ki}\Big]=(1+f)^2tr((\widetilde{Ric})^3)-\frac{R}{6}f(1+f)|\widetilde{Ric}|^2.
\end{equation*}

Using
\begin{equation*}
\begin{split}
&\int_{M^n}\sum_k\nabla_k|\widetilde{Ric}|^2(1+f)f_kdv_g=-\int_{M^n}\sum_k|\widetilde{Ric}|^2((1+f)f_k)_kdv_g\\
&=-\int_{M^n}|\widetilde{Ric}|^2(|\nabla f|^2+(1+f)\triangle f)dv_g=-\int_{M^n}|\widetilde{Ric}|^2(|\nabla f|^2-\frac{R}{2}f(1+f))dv_g
\end{split}
\end{equation*}
and
\begin{equation*}
\begin{split}
&\int_{M^n}f(1+f)|\widetilde{Ric}|^2|dv_g=\int_{M^n}f(1+f)|\widetilde{Ric}|^2|dv_g+\int_{M^n}(1+f)|\widetilde{Ric}|^2|dv_g\\
&=\int_{M^n}(1+f)^2|\widetilde{Ric}|^2|dv_g,
\end{split}
\end{equation*}
we get
\begin{equation}\label{in-4-1}
\begin{split}
&0=\int_{M^n}div(Z)dv_g\\
&=\int_{M^n}\Big[\frac{3}{2}|\widetilde{Ric}|^2|\nabla f|^2-2\sum_{ijk}\overset{\circ}{R}_{ij}\overset{\circ}{R}_{jk}f_if_k
+\frac{R}{12}(1+f)^2|\widetilde{Ric}|^2+(1+f)^2tr((\widetilde{Ric})^3)\Big]dv_g.
\end{split}
\end{equation}
Thus Proposition \ref{pr-4} is proved.
\end{proof}
In order to prove  Theorem \ref{th-3-1} we need the following Lemma.
\begin{lemma}\label{le}
Let $a_1,a_2,\cdots,a_n$  be real numbers satisfying $\sum_{i=1}^{n}a_i=0$.  Then
$$\frac{-(n-2)}{\sqrt{n(n-1)}}\Big(\sum_{i=1}^{n}a_i^2\Big)^{\frac{3}{2}}\leq\sum_{i=1}^{n}a_i^3\leq
\frac{(n-2)}{\sqrt{n(n-1)}}\Big(\sum_{i=1}^{n}a_i^2\Big)^{\frac{3}{2}},$$
and here equality holds  if and only if $(n-1)$ of the $a_i$ are equal.
\end{lemma}
It is easy to prove Lemma \ref{le} by using the method of Lagrange's multipliers
to find the critical points of $\sum_{i=1}^{n}a_i^3$
under  the conditions $\sum_{i=1}^{n}a_i=0$ and $\sum_{i=1}^{n}a_i^2=\beta$.
We omit it here( Or see \cite{okm}).

Now we prove Theorem \ref{th-3-1}. Now we can choose the local orthonormal frame  $\{e_1,e_2,e_3\}$ such that
$$\overset{\circ}{R}_{ij}=a_i\delta_{ij},~~~ i.e., (\overset{\circ}{R}_{ij})=diag(a_1, a_2, a_3).$$
We can assume that $a_1\leq a_2\leq a_3.$ Let $a=\max(a_1^2, a_2^2, a_3^2)$, then $a=a_1^2$ or $a=a_3^2$.
We assume that $a=a_3^2$, then
$$\sum_{ijk}\overset{\circ}{R}_{ij}\overset{\circ}{R}_{jk}f_if_k\leq a_3^2|\nabla f|^2,~~|\widetilde{Ric}|^2=a_1^2+a_2^2+a_3^2.$$
Thus
\begin{equation}\label{cpe-4-2}
\begin{split}
&\frac{3}{2}|\widetilde{Ric}|^2|\nabla f|^2-2\sum_{ijk}\overset{\circ}{R}_{ij}\overset{\circ}{R}_{jk}f_if_k
\geq\Big[\frac{3}{2}(a_1^2+a_2^2)-\frac{1}{2}a_3^2\Big]|\nabla f|^2\\
&=\Big[\frac{3}{2}(a_1^2+a_2^2)-\frac{1}{2}(a_1+a_2)^2\Big]|\nabla f|^2=\Big[a_1^2+a_2^2-a_1a_2\Big]|\nabla f|^2\\
&=\Big[\frac{1}{2}(a_1^2+a_2^2)+\frac{1}{2}(a_1-a_2)^2\Big]|\nabla f|^2\geq 0,
\end{split}
\end{equation}
Here we use $tr(\widetilde{Ric})=a_1+a_2+a_3=0$.

Now If $tr((\widetilde{Ric})^3)\geq \frac{-R}{12}|\widetilde{Ric}|^2$, then
$$\int_{M^n}\Big[\frac{R}{12}(1+f)^2|\widetilde{Ric}|^2+(1+f)^2tr((\widetilde{Ric})^3)\Big]dv_g\geq 0.$$
then from (\ref{in-4-1}) we can obtain
$|\widetilde{Ric}|^2=0$. Thus we have $\widetilde{Ric}=0$, and  Theorem \ref{th-3-1} is proved.

We use Lemma \ref{le} and Theorem \ref{th-3-1} to prove Theorem \ref{th-4-1}.
From Lemma \ref{le}, we obtain
\begin{equation}\label{niden}
tr((\widetilde{Ric})^3)\geq \frac{-|\widetilde{Ric}|}{\sqrt{6}}|\widetilde{Ric}|^2.
\end{equation}
Now if $|\widetilde{Ric}|^2\leq \frac{R^2}{24},$, then
\begin{equation}\label{niden}
tr((\widetilde{Ric})^3)\geq \frac{-|\widetilde{Ric}|}{\sqrt{12}}|\widetilde{Ric}|^2\geq \frac{-R}{12}|\widetilde{Ric}|^2.
\end{equation}
Combining Theorem \ref{th-3-1}, we finish the proof of Theorem \ref{th-4-1}.

To prove Theorem \ref{th-3-2}, we need the following proposition,
\begin{PROPOSITION}\label{pr-4}
Let $(M^3, g, f)$ be a $3$-dimensional closed, oriented, connected CPE  metric, then
\begin{equation}\label{ricci-3}
\begin{split}
&\int_{M^3}\Big[\frac{7}{2}|\widetilde{Ric}|^2|\nabla f|^2+\sum_{ijk}(1+f)^2|\overset{\circ}{R}_{ij,k}|^2
+\frac{5R}{4}(1+f)^2|\widetilde{Ric}|^2)\Big]dv_g\\
&=\int_{M^3}6\Big[\sum_{ijk}\overset{\circ}{R}_{ij}\overset{\circ}{R}_{jk}f_if_k
-(1+f)^2tr((\widetilde{Ric})^3\Big]dv_g.
\end{split}
\end{equation}
\end{PROPOSITION}
\begin{proof}
Using Ricci identity, (\ref{cpe-4}) and (\ref{ricci-2}), we calculate the following equation,
\begin{equation*}
\begin{split}
\triangle(\overline{R}_{ij})&:=\sum_k\overline{R}_{ij,kk}=\sum_k\Big[\overline{R}_{ik,j}+\overline{R}_{ij,k}-\overline{R}_{ik,j}\Big]_k\\
&=\sum_k\overline{R}_{ik,jk}+\sum_k\Big[\overline{R}_{ij,k}-\overline{R}_{ik,j}\Big]_k\\
&=\sum_{km}\Big[\overline{R}_{ik,kj}+\overline{R}_{mk}R_{mijk}+\overline{R}_{im}R_{mkjk}\Big]\\
&+\sum_{km}\Big[f_m\overset{\circ}{R}_{mj}\delta_{ik}-f_m\overset{\circ}{R}_{mk}\delta_{ij}+f_j\overset{\circ}{R}_{ik}-f_k\overset{\circ}{R}_{ij}\Big]_k\\
&=\sum_{k}6(1+f)\overset{\circ}{R}_{im}\overset{\circ}{R}_{jm}-2(1+f)|\widetilde{Ric}|^2\delta_{ij}\\
&+\frac{R}{2}(1+f)\overset{\circ}{R}_{ij}+\sum_kf_k(\overset{\circ}{R}_{ik,j}+\overset{\circ}{R}_{jk,i}-\overset{\circ}{R}_{ij,k}),
\end{split}
\end{equation*}
Here we use $\sum_j\overline{R}_{ij,j}=\sum_jf_j\overset{\circ}{R}_{ij}$.
Using
\begin{equation*}
\begin{split}
&\int_{M^3}\sum_k\nabla_k|\widetilde{Ric}|^2(1+f)f_kdv_g=-\int_{M^3}\sum_k|\widetilde{Ric}|^2((1+f)f_k)_kdv_g\\
&=-\int_{M^3}|\widetilde{Ric}|^2(|\nabla f|^2+(1+f)\triangle f)dv_g=-\int_{M^3}|\widetilde{Ric}|^2(|\nabla f|^2-\frac{R}{2}f(1+f))dv_g
\end{split}
\end{equation*}
and
\begin{equation*}
\begin{split}
&\int_{M^3}f(1+f)|\widetilde{Ric}|^2|dv_g=\int_{M^3}f(1+f)|\widetilde{Ric}|^2|dv_g+\int_{M^3}(1+f)|\widetilde{Ric}|^2|dv_g\\
&=\int_{M^3}(1+f)^2|\widetilde{Ric}|^2|dv_g,
\end{split}
\end{equation*}
 we can calculate the equations
\begin{equation*}
\begin{split}
&\int_{M^3}\sum_{ijk}|\overline{R}_{ij,k}|^2dv_g=\int_{M^3}\Big[\sum_{ijk}(f_k\overset{\circ}{R}_{ij}+\overline{R}_{ij,k})^2\Big]dv_g\\
&=\int_{M^3}\Big[|\widetilde{Ric}|^2|\nabla f|^2+(1+f)^2\sum_{ijk}|\overset{\circ}{R}_{ij,k}|^2+2\sum_k\nabla_k|\widetilde{Ric}|^2(1+f)f_k\Big]dv_g\\
&=\int_{M^3}\Big[|\widetilde{Ric}|^2|\nabla f|^2+(1+f)^2\sum_{ijk}|\overset{\circ}{R}_{ij,k}|^2-|\widetilde{Ric}|^2(|\nabla f|^2-\frac{R}{2}f(1+f))\Big]dv_g\\
&=\int_{M^3}\Big[(1+f)^2\sum_{ijk}|\overset{\circ}{R}_{ij,k}|^2+\frac{R}{2}f(1+f)|\widetilde{Ric}|^2\Big]dv_g,
\end{split}
\end{equation*}
and
\begin{equation*}
\begin{split}
&\int_{M^3}\Big[\sum_{ijk}\overline{R}_{ij}f_k(\overset{\circ}{R}_{ik,j}+\overset{\circ}{R}_{jk,i}-\overset{\circ}{R}_{ij,k})\Big]dv_g\\
&=\int_{M^3}\Big[\sum_{ijk}\overset{\circ}{R}_{ij}f_k\Big((1+f)\overset{\circ}{R}_{ik,j}
+(1+f)\overset{\circ}{R}_{jk,i}-(1+f)\overset{\circ}{R}_{ij,k}\Big)\Big]dv_g\\
&=\int_{M^3}\Big[\sum_{ijk}\overset{\circ}{R}_{ij}f_k\Big((1+f)[\overset{\circ}{R}_{ik,j}-\overset{\circ}{R}_{ij,k}]
+(1+f)[\overset{\circ}{R}_{jk,i}-\overset{\circ}{R}_{ij,k}]+(1+f)\overset{\circ}{R}_{ij,k}\Big)\Big]dv_g\\
&=\int_{M^3}\Big[\frac{7}{2}|\widetilde{Ric}|^2|\nabla f|^2+\frac{R}{4}(1+f)^2|\widetilde{Ric}|^2-6\sum_{ijk}
\overset{\circ}{R}_{ij}\overset{\circ}{R}_{jk}f_if_k\Big]dv_g.
\end{split}
\end{equation*}

Thus we can obtain the following equation,
\begin{equation*}
\begin{split}
&0=\int_{M^3}\Big[\frac{1}{2}\triangle((\overline{R}_{ij})^2\Big]dv_g
=\int_{M^3}\Big[\sum_{ijk}|\overline{R}_{ij,k}|^2+\sum_{ij}\overline{R}_{ij}\triangle(\overline{R}_{ij})\Big]dv_g\\
&=\int_{M^3}\Big[\frac{7}{2}|\widetilde{Ric}|^2|\nabla f|^2+\sum_{ijk}(1+f)^2|\overset{\circ}{R}_{ij,k}|^2
+\frac{5R}{4}(1+f)^2|\widetilde{Ric}|^2+6(1+f)^2tr((\widetilde{Ric})^3)\\
&-6\sum_{ijk}\overset{\circ}{R}_{ij}\overset{\circ}{R}_{jk}f_if_k
\Big]dv_g.
\end{split}
\end{equation*}
Thus Proposition \ref{pr-4} is proved.
\end{proof}

Next we prove Theorem \ref{th-3-2}. We can choose the local orthonormal frame  $\{e_1,e_2,e_3\}$ such that
$$\overset{\circ}{R}_{ij}=a_i\delta_{ij},~~~ i.e., (\overset{\circ}{R}_{ij})=diag(a_1, a_2, a_3).$$
we can assume that $a_1\leq a_2\leq a_3.$ Let $a=\max(a_1^2, a_2^2,  a_3^2)$, then $a=a_1^2$ or $a=a_3^2$.
We assume that $a=a_3^2$, then
$$\sum_{ijk}\overset{\circ}{R}_{ij}\overset{\circ}{R}_{jk}f_if_k\leq a_3^2|\nabla f|^2,~~|\widetilde{Ric}|^2=\sum_ma_m^2.$$
Thus
\begin{equation}\label{cpe-4-2}
\begin{split}
&\frac{7}{2}|\widetilde{Ric}|^2|\nabla f|^2-6\sum_{ijk}\overset{\circ}{R}_{ij}\overset{\circ}{R}_{jk}f_if_k
\geq\Big[\frac{7}{2}(a_1^2+a_2^2+a_3^2)-6a_3^2\Big]|\nabla f|^2\\
&=\Big[\frac{7}{2}(a_1^2+a_2^2)-\frac{5}{2}(a_1+a_2)^2\Big]|\nabla f|^2=\Big[(a_1-a_2)^2-3a_1a_2\Big]|\nabla f|^2.\\
\end{split}
\end{equation}
Here we use $tr(\widetilde{Ric})=a_1+a_2+a_3=0$.

Since
$$tr((\widetilde{Ric})^3)=a_1^3+a_2^3+a_3^3=3a_1a_2a_3\leq 0$$
 and $a_1\leq a_2\leq a_3$, then we have
 $$a_1\leq 0,~~0\leq a_2\leq a_3.$$
So $a_1a_2\leq 0$ and
$$\Big[(a_1-a_2)^2-3a_1a_2\Big]|\nabla f|^2\geq 0.$$
Thus
\begin{equation}\label{int3-1}
\frac{7}{2}|\widetilde{Ric}|^2|\nabla f|^2-6\sum_{ijk}\overset{\circ}{R}_{ij}\overset{\circ}{R}_{jk}f_if_k\geq 0.
\end{equation}

Now if $$\frac{-5R}{24}|\widetilde{Ric}|^2\leq tr((\widetilde{Ric})^3)\leq 0,$$ then
\begin{equation}\label{int3-2}
\int_{M^3}\Big[\frac{5R}{4}(1+f)^2|\widetilde{Ric}|^2+6(1+f)^2tr((\widetilde{Ric})^3)
\Big]dv_g\geq 0.
\end{equation}
Combining (\ref{ricci-3}), (\ref{int3-1}) and (\ref{int3-2}), we obtain that $\widetilde{Ric}=0$. Thus we finish the proof of Theorem \ref{th-3-2}.

There  are some special properties of traceless tensor on $3$-dimensional manifold, as shown in the following equations,
\begin{equation*}
\begin{split}
&tr((\widetilde{Ric})^4)=\frac{1}{2}|\widetilde{Ric}|^4, ~tr((\widetilde{Ric})^5)=\frac{5}{6}|\widetilde{Ric}|^2tr((\widetilde{Ric})^3),\\
&tr((\widetilde{Ric})^6)=\frac{1}{4}|\widetilde{Ric}|^6+\frac{1}{3}\Big(tr((\widetilde{Ric})^3)\Big)^2.
\end{split}
\end{equation*}
Below, we establish three interesting geometric integral identities on CPE manifolds. Now we define the vector fields
\begin{equation*}
\begin{split}
&X=\sum_{ijkl}\overset{\circ}{R}_{ij}\overset{\circ}{R}_{jk}\overline{R}_{kl}f_le_k,\\
&Y=\sum_{ijklm}\overset{\circ}{R}_{ij}\overset{\circ}{R}_{jk}\overset{\circ}{R}_{kl}\overline{R}_{lm}f_me_l,~~
Z=\sum_{ijklmt}\overset{\circ}{R}_{ij}\overset{\circ}{R}_{jk}\overset{\circ}{R}_{kl}\overset{\circ}{R}_{lm}\overline{R}_{mt}f_te_m.
\end{split}
\end{equation*}
By calculating the divergence of these three vector fields  and utilizing the specificity of three-dimensional manifolds, we can obtain the following three computational identities on $3$-dimension CPE manifold.
\begin{equation}\label{int-5}
\int_{M^3}\Big[\frac{5}{2}|\widetilde{Ric}|^2\sum_{ij}\overset{\circ}{R}_{ij}f_if_j-
5\sum_{ijkl}\overset{\circ}{R}_{ij}\overset{\circ}{R}_{jk}\overset{\circ}{R}_{kl}f_if_l+\frac{5}{3}tr((\widetilde{Ric})^3|\nabla f|^2\Big]dv_g=0.
\end{equation}
\begin{equation}\label{int-6}
\begin{split}
\int_{M^3}\Big[4|\widetilde{Ric}|^2\sum_{ijk}\overset{\circ}{R}_{ij}\overset{\circ}{R}_{jk}f_if_k-
8\sum_{ijklm}\overset{\circ}{R}_{ij}\overset{\circ}{R}_{jk}\overset{\circ}{R}_{kl}\overset{\circ}{R}_{lm}f_if_m\\
+\frac{8}{3}tr((\widetilde{Ric})^3|\sum_{ij}\overset{\circ}{R}_{ij}f_if_j\Big]dv_g=0.
\end{split}
\end{equation}
\begin{equation}\label{int-7}
\begin{split}
\int_{M^3}\Big[\frac{11}{2}|\widetilde{Ric}|^2\sum_{ijkl}\overset{\circ}{R}_{ij}\overset{\circ}{R}_{jk}\overset{\circ}{R}_{kl}f_if_l-
11\sum_{ijklmt}\overset{\circ}{R}_{ij}\overset{\circ}{R}_{jk}\overset{\circ}{R}_{kl}\overset{\circ}{R}_{lm}\overset{\circ}{R}_{mt}f_if_t\\
+\frac{11}{3}tr((\widetilde{Ric})^3|\sum_{ijk}\overset{\circ}{R}_{ij}\overset{\circ}{R}_{jk}f_if_k\Big]dv_g=0.
\end{split}
\end{equation}\\

These integral identities reflect certain rigidity of 3-dimensional CPE manifolds and are expected to be applied in the study of CPE metric.

{\bf Acknowledgements:} Authors  would like to thank Prof. Wei Yuan for some useful discussions.
Authors are partially supported by the grant No.12071028 of NSFC.


\begin{thebibliography}{11}
\bibitem{bal-0}H. Baltazar, A. Barros, R. Batista and E. Viana, \textit{On static manifolds and related critical spaces with zero radial Weyl curvature,} Monatsh. Math., Vol.191, (2020), 449-463.

\bibitem{Bal}H. Baltazar, \textit{On critical point equation of compact manifolds with zero radial Weyl curvature,} Geom. Dedic., Vol.202, (2019), 337-355.

\bibitem{B-R} A. Barros and E.Ribeiro Jr, \textit{Critical point equation on four-dimensional compact manifolds}, Math. Nachr., Vol.287, (2013), 1618-1623.

\bibitem{bal-1}H. Baltazar and E. Ribeiro Jr, \textit{Remarks on critical metrics of the scalar curvature and volume functionals on compact manifolds with boundary,}
Pacific J. Math., Vol.297, (2018),  29-45.

\bibitem{Besse} A.L.Besse, \textit{Einstein manifolds}, Springer-Verlag, New York, 1987.

\bibitem{F_Y}Y. Fang and W. Yuan, \textit{Brown-York mass and positive scalar curvature II-Besse's conjecture and related problems,} Ann. Global Anal. Geom., Vol.56, (2019), 1-15.

\bibitem{He} H. Y. He, \textit{Critical metrics of the volume functional on three-dimensional manifolds}, Math. Nachr., Vol.296, (2023) 2838-2849.

\bibitem{C-H} J. Chang and S. Hwang, \textit{Critical points and warped product metrics}, Bull. Korean Math. Soc., Vol.41, (2004), 117-123.

\bibitem{C-H-Y_1} J. Chang, S. Hwang, and G. Yun, \textit{Critical point equation of the total scalar curvature functional}, Bull. Korean Math. Soc., Vol.49, (2012), 655-667.

\bibitem{C-H-Y_2} J .Chang, S. Hwang, and G. Yun, \textit{Rigidity of the critical point equation}, Math. Nachr., Vol.283, (2010), 846-853.

\bibitem{C-H-Y_3-1} J. Chang, S. Hwang, and G. Yun, \textit{Total Scalar Curvature and Harmonic Curvature}, Taiwanese J. Math., Vol.18, (2014), 1439-1458.

\bibitem{C-H-Y_3}J. Chang, S. Hwang, and G. Yun, \textit{Erratum to: total scalar curvature and harmonic curvature,} Taiwanese J. Math., Vol.20, (2016), 699-703.


\bibitem{Hwang_1} S. Hwang, \textit{Critical points of the scalar curvature functional on the space of metrics of constant scalar curvature}, manuscripta math., Vol.103, (2000), 135-142.

\bibitem{Hwang_2} S. Hwang, \textit{The criticlal point equation on a three-dimensional compact manifold}, Proc. Amer. Math. Soc., Vol.131, (2003), 3221-3230.

\bibitem{Hwang_3} S. Hwang, \textit{Three dimensional criticlal point of the total scalar curvature}, Bull. Korean Math. Soc., Vol.50, (2013), 867- 871.

\bibitem{Koiso} N. Koiso, \textit{A decomposition of the space of Riemannian metrics on a manifold}, Osaka J. Math., Vol.16, (1979), 423-429.

\bibitem{Lafontaine} J. Lafontaine, \textit{Sur la g\'eom\'etrie d'une g\'en\'eralisation de
l'\'equation diff\'erentielle d'Obata}, J. Math. Pures Appliqu\'ees, Vol.62, (1983) 63-72.

\bibitem{bl}B. Leandro, \textit{A note on critical point metrics of the total scalar curvature functional,} J. Math. Anal. Appl., Vol.424, (2015), 1544-1548.

\bibitem{Obata} M. Obata, \textit{Certain conditions for a Riemannian manifold to be isometric with a
sphere}, J. Math. Soc. Japan., Vol.14, (1962), 333-340.


\bibitem{Q-Y} J. Qing and W. Yuan, \textit{A note on static spaces and related problems}, J. Geom. Phys., Vol.74, (2013), 18-27.

\bibitem{okm}Okumura, Masafumi, \textit{Hypersurfaces and a pinching problem on the second fundamental tensor,} Amer. J. Math., Vol.96,
(1974) 207-213.


\end{thebibliography}
\end{document}